\documentclass[oneside]{amsart}

\usepackage[utf8]{inputenc}
\usepackage{amsthm,mathtools,tikz,float,caption}
\usetikzlibrary{patterns,matrix}
\usepackage{hyperref}

\usepackage[natbib=true,style=numeric,maxnames=10]{biblatex}
\usepackage[babel]{csquotes}
\bibliography{paper-macneille-reals.bib}

\newcommand{\NN}{\mathbb{N}}
\newcommand{\RR}{\mathbb{R}}

\title[]{A constructive Knaster--Tarski proof of the uncountability of the reals}

\author{Ingo Blechschmidt}
\address{Università di Verona \\
Department of Computer Science \\
Strada le Grazie 15 \\
37134 Verona, Italy}
\email{ingo.blechschmidt@univr.it}

\author{Matthias Hutzler}
\address{Universität Augsburg \\
Institut für Mathematik \\
Universitätsstr. 14 \\
86159 Augsburg, Germany}
\email{matthias.ralph.hutzler@student.uni-augsburg.de}

\theoremstyle{definition}

\theoremstyle{plain}

\newtheorem*{thm*}{Theorem}

\theoremstyle{remark}

\newcommand{\defeq}{\vcentcolon=}

\begin{document}

\begin{abstract}
  We give an uncountability proof of the reals which relies on their
  order completeness instead of their sequential completeness. We use neither a
  form of the axiom of choice nor the law of excluded middle, therefore
  the proof applies to the MacNeille reals in any flavor of constructive
  mathematics. The proof leans heavily on Levy's unusual proof of the
  uncountability of the reals.
\end{abstract}

\maketitle
\thispagestyle{empty}

One way to verify the uncountability of the reals is as follows. We first
observe that the usual diagonalization technique shows that the powerset of the
naturals is uncountable. We then show that this powerset is in bijection with
the reals. In constructive mathematics, more precisely the kind of mathematics
which can be carried out in any topos or the kind of mathematics formalizable
in Intuitionistic Zermelo--Fraenkel set theory, the first step is still valid, while the
second might fail. This failure may occur for any of the several possible flavors of the reals
such as the Cauchy reals, the Dedekind reals or the MacNeille reals (which all
coincide in classical mathematics, but might differ in constructive
mathematics). Therefore a different approach is needed.

This note gives a constructive proof that one of these flavors,
the MacNeille reals, is uncountable. To the best of our knowledge, this
is the first result in that direction. However, it is just a baby step towards
an understanding whether any of the more interesting flavors of the reals can constructively be
shown to be uncountable, a problem posed to us by Andrej Bauer who we
gratefully acknowledge. The proof presented here is made possible because the
MacNeille reals -- unlike the Cauchy or Dedekind reals -- can constructively be
shown to be (conditionally) order complete~\cite[Lemma~D4.7.7]{johnstone:elephant}.

Sensibilities of constructive mathematics aside, the proof presented here is
interesting because it uses only the order completeness of the reals, not their
sequential completeness, and because it puts the Knaster--Tarski fixed point
theorem to good use. This fixed point theorem is fundamental to theoretical
computer science, but appears to be seldomly used in classical analysis.
The proof is an adaptation of Levy's unusual proof~\cite{levy:uncountability}.

\begin{thm*}Let~$f : \NN \to \RR$ be a map. Then there is a number~$x_0 \in \RR$
such that for no~$n_0 \in \NN$, $f(n_0) = x_0$.\end{thm*}

\begin{proof}The map
\[ g : \RR \longrightarrow \RR,\
  x \longmapsto \sup_M \sum_{n \in M} 2^{-n}, \]
where~$M$ ranges over all those (Bishop-)finite subsets of~$\NN$ such
that~$f[M] < x$, is well-defined (because the sets the suprema are taken of are
inhabited by zero and bounded from above by~$2$), monotone, has a postfixpoint
($0 \leq g(0)$), and has an upper bound for all of its postfixpoints (if~$x
\leq g(x)$, then~$x \leq 2$). By the Knaster--Tarski fixed point theorem, it
therefore has a greatest postfixpoint~$x_0$.

If~$x_0 = f(n_0)$ for a number~$n_0 \in \NN$, then for any finite
subset~$M$ of~$\NN$ such that~$f[M] < x_0$,
\[ \sum_{n \in M} 2^{-n} + 2^{-n_0} = \sum_{n \in M \cup \{n_0\}} 2^{-n}
  \leq g(x_0 + 2^{-n_0}), \]
hence~$g(x_0 + 2^{-n_0}) \geq g(x_0) + 2^{-n_0} \geq x_0 + 2^{-n_0}$.
Thus~$x_0 + 2^{-n_0}$ is a greater postfixpoint than~$x_0$, a contradiction.
\end{proof}

It is possible to unwind the application of the Knaster--Tarski fixed point
theorem to obtain an entirely elementary proof of uncountability. This
unwinding makes the impredicative nature of the proof manifest.

\begin{proof}[Second proof]We consider the same map~$g : \RR \to \RR$ as in the
first proof. Let~$x_0$ be the supremum of the set~$A \defeq \{ x \in \RR \,|\,
x \leq g(x) \}$; this supremum exists because~$A$ is inhabited (by zero)
and bounded from above (by~$2$).

If~$x_0 = f(n_0)$ for a number~$n_0 \in \NN$, then~$x_0 - 2^{-n_0}$ is an
upper bound for~$A$, contradicting the fact that~$x_0$ is the least upper bound
of~$A$: Let~$x \in A$. If~$x > x_0 - 2^{-n_0}$, then~$g(x +
2^{-n_0}) \geq g(x) + 2^{-n_0} \geq x + 2^{-n_0}$ (where the first inequality
is as in the first proof, exploiting that~$x \leq x_0$ by definition of~$x_0$), hence~$x + 2^{-n_0} \in A$,
thus~$x + 2^{-n_0} \leq x_0$, a contradiction. Hence~$x \leq x_0 - 2^{-n_0}$.
\end{proof}

\printbibliography

\end{document}